\documentclass[12pt, reqno]{amsart}
\usepackage{amsmath, amsthm, amscd, amsfonts, amssymb, graphicx, xcolor, enumerate}
\usepackage{multicol, lipsum, paracol, diagbox, tabularx, multirow, makecell}

\usepackage{dsfont, booktabs}

\usepackage[bookmarksnumbered, colorlinks, plainpages]{hyperref}

\theoremstyle{definition}
\newtheorem{theorem}{Theorem}[section]
\newtheorem{lemma}[theorem]{Lemma}

\newtheorem{corollary}[theorem]{Corollary}
\newtheorem{definition}[theorem]{Definition}

\newtheorem{conjecture}[theorem]{Conjecture}

\theoremstyle{remark}
\newtheorem{remark}[theorem]{Remark}
\numberwithin{equation}{section}

\newcommand{\BC}{\mathbb{C}}

\newcommand{\BF}{\mathbb{F}}

\newcommand{\BK}{\mathbb{K}}

\newcommand{\BR}{\mathbb{R}}

\newcommand{\BZ}{\mathbb{Z}}

\newcommand{\CN}{\mathcal{N}}

\newcommand{\gs}{\geqslant}
\newcommand{\ls}{\leqslant}
\newcommand{\pt}[1]{\left(#1\right)}
\newcommand{\cb}[1]{\left\{#1\right\}}
\newcommand{\jk}[1]{\left\langle#1\right\rangle}
\newcommand{\fk}[1]{\left[#1\right]}
\newcommand{\abs}[1]{\left|#1\right|}

\newcommand{\fC}{\mathfrak{C}}
\newcommand{\fS}{\mathfrak{S}}

\newcommand{\GL}{\mathrm{GL}}

\newcommand{\Cay}{\mathrm{Cay}}

\newcommand{\jz}[4]{{\begin{pmatrix}
		#1&#2\\
		#3&#4
\end{pmatrix}}}

\newcommand{\one}{\mathds{1}}
\newcommand{\lcm}{\mathrm{lcm}}
\newcommand{\Sym}{\mathrm{Sym}}

\begin{document}
	\setcounter{page}{1}

	\centerline{}
	
	\centerline{}

	%------------------------------------------------------------------------------
	
	%Title of the paper
	\title[]{A\lowercase{n} E\lowercase{rd\H{o}s}-K\lowercase{o}-R\lowercase{ado result for some principal series representations}}
	
	%Author names and affiliations
	\author[]{J\lowercase{iaqi} L\lowercase{iao}$^1$ \qquad G\lowercase{uiying} Y\lowercase{an}$^2$$^*$}
	
	%\address{$^{1}$ Department of Mathematics, University of California, San Diego, USA.}
	%\email{\textcolor[rgb]{0.00,0.00,0.84}{first@amcs.org}}
	
	%\address{$^{2}$ Department of Computer Science, Moscow State University, Moscow, Russia}
	%\email{\textcolor[rgb]{0.00,0.00,0.84}{second@ieee.org}}
	
\thanks{State Key Laboratory of Mathematical Sciences, Academy of Mathematics and Systems Science, Chinese Academy of Sciences, Beijing 100190, China \& School of Mathematical Sciences, University of Chinese Academy of Sciences, Beijing 100049, China. Emails: 975497560@qq.com$^1$, yangy@amss.ac.cn$^2$$^*$. Supported by The Key Program of National Natural Science Fund No. 12231018 and National Key R\&D Program of China No. 2023YFA1009600 }

\begin{abstract}

Let $V$ be an irreducible principal series representation of $\GL_2(q)$ satisfying certain conditions. Two subsets $S_1, S_2 \subseteq \GL_2(q)$ are called \emph{cross-$t$-intersecting} if $\dim\cb{v \in V: g_1v = g_2v} \gs t$ for any $\pt{g_1, g_2} \in S_1 \times S_2$. In this paper, we determine $\max\pt{\abs{S_1}\cdot\abs{S_2}}$ where $S_1, S_2 \subseteq \GL_2(q)$ are cross-$1$-intersecting. Our proof combines eigenvalue techniques with the representation theory of $\GL_2(q)$.

\end{abstract} \maketitle
	
\section{Introduction}

The classical Erd\H{o}s-Ko-Rado theorem \cite{MR140419} states that if $n \gs 2k$, an intersecting family of $k$-subsets of $\cb{1, 2, \ldots, n}$ has size at most $\binom{n - 1}{k - 1}$; if equality holds, the family must consist of all $k$-subsets containing a fixed element. We deal with analogues of this result for $\BC$-linear representations of finite groups.

Let $G$ be a finite group. If $\alpha: G \rightarrow \Sym(X)$ is a transitive permutation representation, where $X$ is a finite set, we say that two subsets $S_1, S_2 \subseteq G$ are \emph{cross-$t$-intersecting} (with respect to $\alpha$) if $\abs{\cb{x \in X: \alpha(g_1)(x) = \alpha(g_2)(x)}} \gs t$ for any $(g_1, g_2) \in S_1 \times S_2$. In particular, $S_1$ is said to be \emph{$t$-intersecting} (with respect to $\alpha$) if $S_1 = S_2$. Similarly, if $\beta: G \rightarrow \GL(V)$ is an irreducible linear representation, where $V$ is a finite-dimensional vector space, we say that two subsets $S_1, S_2 \subseteq G$ are \emph{cross-$t$-intersecting} (with respect to $\beta$) if $\dim\cb{v \in V: \beta(g_1)(v) = \beta(g_2)(v)} \gs t$ for any $(g_1, g_2) \in S_1 \times S_2$. In particular, $S_1$ is said to be \emph{$t$-intersecting} (with respect to $\beta$) if $S_1 = S_2$. One can ask, for each representation, what is the maximum possible product of sizes of a pair of cross-$t$-intersecting subsets of $G$.

\subsection{History}

On the intersection problem with respect to permutation representation, Ellis, Friedgut and Pilpel \cite{MR2784326} completely solved the case for $(G, X) = (\fS_n, \cb{1, 2, \ldots, n})$ with $n \gg t$, which was once the longstanding Deza-Frankl conjecture \cite{MR439648}. It was shown in \cite{MR2784326} that, for each fixed $t$, and all sufficiently large $n$, every pair of cross-$t$-intersecting subsets $S_1, S_2 \subseteq \fS_n$ have product of sizes at most $\pt{(n - t)!}^2$ and equality holds if and only if $S_1 = S_2$ are cosets of the stabilizer of a $t$-tuple of distinct points in $\cb{1, 2, \ldots, n}$. Ellis and Lifshitz \cite{MR4484211} use junta method to prove a stronger conclusion for the case $S_1 = S_2 \subseteq \fS_n$. 

On the intersection problem with respect to linear representation, Ernst and Schmidt \cite{MR4600198} partly solved the case for $(G, V) = (\GL_n(q), \BF_q^n)$ with $n \gg t$, which generalizes Meagher and Razafimahatratra's result \cite{MR4521814}. Recall that the \emph{characteristic vector} $\overrightarrow{S}$ of $S \subseteq G$ is a length-$\abs{G}$ real column vector with entries indexed by the elements of $G$, and the $g$-entry of $\overrightarrow{S}$ is $1$ if $g \in S$ and $0$ otherwise. In \cite{MR4600198}, they show that for each fixed $t$, and all sufficiently large $n$, every pair of cross-$t$-intersecting subsets $S_1, S_2 \subseteq \GL_n(q)$ have product of sizes at most $\pt{\prod_{i = t}^{n - 1}\pt{q^n - q^i}}^2$, and if equality holds, then $\overrightarrow{S_1}$ and $\overrightarrow{S_2}$ are linear combinations of the characteristic vectors of cosets of the stabilizers of a $t$-tuple of linearly independent vectors in $\BF_q^n$. Ellis, Kindler and Lifshitz \cite{MR4684861} also use junta method to prove a stronger conclusion for the case $S_1 = S_2 \subseteq \GL_n(q)$. For further  information, we recommend readers refer to a well-written textbook \cite{MR3497070} and two comprehensive surveys \cite{MR4421401, MR3534067}.

\subsection{Our main result}

Fix a generator $\varepsilon$ of $\BF_q^\times$. Define multiplicative characters $\mu_i: \BF_q^\times \rightarrow \BC^\times$ by $\mu_i(\varepsilon) = \exp\pt{2\cdot\pi\cdot\sqrt{-1}\cdot i/(q - 1)}$. Let $V_{m, n}$ be the irreducible principal series representation of $\GL_2(q)$ induced by $\mu_m$ and $\mu_n$ where $m \ne n$, which is one of the four classes of irreducible $\BC$-representations of $\GL_2(q)$ (see \ref{representation_theory}). Our main result is as follow.

%\newpage
\begin{theorem}\label{main_thm}
	
	Fix an odd prime power $q$ and factor $q - 1 = \ell \cdot m$ where $\ell$ is an odd prime with $\ell \nmid m$. Let $\beta: \GL_2(q) \rightarrow \GL(V_{m, 0})$ be the irreducible principal series representation. If two subsets $S_1, S_2 \subseteq \GL_2(q)$ are cross-$1$-intersecting with respect to $\beta$, then $\max_{(S_1, S_2)}\sqrt{\abs{S_1}\cdot\abs{S_2}} = \abs{\GL_2(q)}/\ell$.	
	
\end{theorem}

Theorem \ref{main_thm} may be seen as a $\BC$-analog of \cite[Theorem 1.2]{MR4521814}. Let $\log_\varepsilon: \BF_q^\times \rightarrow \BZ/(q - 1)\BZ$ be the discrete logarithm. We therefore pose the following conjecture.

\begin{conjecture}
	
	The hypothesis are the same as in Theorem \ref{main_thm}. If two subsets $S_1, S_2 \subseteq \GL_2(q)$ are cross-$1$-intersecting with respect to $\beta$ as well as $\sqrt{\abs{S_1}\cdot\abs{S_2}} = \abs{\GL_2(q)}/\ell$, then $S_1 = S_2 = H:= \cb{g \in \GL_2(q) \;|\;  \log_\varepsilon(\det(g)) \equiv 0 \pmod{\ell}}$.		
	
\end{conjecture}

\subsection{Structure of the paper}

Section \ref{Notation} collects almost all the notation that appears in this paper and section \ref{Background} provides the background from spectral graph theory. In section \ref{Proof} we prove Theorem \ref{main_thm}.

\newpage
\section{Notation}\label{Notation}

\[\begin{tabular}{p{2.3cm}p{10.7cm}}
	\toprule[0.1em]
	
	Symbol & Meaning \\
	
	\midrule[0.1em]
	
	%$p$ & an odd prime \\
	
	$q$  & an odd prime power  \\
	
	$\ell$ & an odd prime factor of $q - 1$ with multiplicity one \\
	
	$\BF_q$  & a finite field with $q$ elements  \\
	
	$\varepsilon$ & a fixed generator of $\BF_q^\times$ \\
	
	$\BK_q$  & $\BF_q\fk{\sqrt{\varepsilon}}$  \\
	
	$\omega$ & a fixed generator of $\BK_q^\times$ with $\omega^{q + 1} = \varepsilon$ \\
	
	$\log_\varepsilon$  & the discrete logarithm,  $\log_\varepsilon: \BF_q^\times \rightarrow \BZ/(q - 1)\BZ$  \\
	
	$\log_\omega$ & the discrete logarithm,  $\log_\omega: \BK_q^\times \rightarrow \BZ/(q^2 - 1)\BZ$  \\
	
	$\CN$ & the norm map $\BK_q \rightarrow \BF_q$, $\CN(x) = x^q \cdot x = x^{q + 1}$ \\
	
	$\GL_2(q)$ & the general linear group over $\BF_q$ of degree two \\
	
	$B$ & the subgroup of $2 \times 2$ upper triangular invertible matrices \\
	
	$H$ & $\cb{g \in \GL_2(q) \;|\; \log_\varepsilon(\det(g)) \equiv 0 \pmod{\ell}}$ \\
	
	$\Theta$ & a fixed embedding $\BK_q^\times \hookrightarrow \GL_2(q)$, $\Theta(x + y\sqrt{\varepsilon}) = \jz{x}{y}{y\varepsilon}{x}$ \\
	
	$\zeta$  & $\exp\pt{{2\cdot\pi\cdot\sqrt{-1}}/\pt{q - 1}}$  \\
	
	$\mu_i$ & a homomorphism $\BF_q^\times \rightarrow \BC^\times$, defined by $\mu_i(\varepsilon^j) = \zeta^{i \cdot j}$ \\
	
	$\xi$  & $\exp\pt{{2\cdot\pi\cdot\sqrt{-1}}/\pt{q^2 - 1}}$  \\
	
	$\lambda_i$ & a homomorphism $\BK_q^\times \rightarrow \BC^\times$, defined by $\lambda_i(\omega^j) = \xi^{i \cdot j}$ \\
	
	$V_{m, n}$ & $\cb{f: \GL_2(q) \rightarrow \BC \;|\;  f\pt{\jz{a}{b}{}{d}x} = \mu_m(a)\cdot\mu_n(d)\cdot f(x)}$ \\
	
	$\one$ & $\one(x) = 1$ if and only if $x = 1$, otherwise $\one(x) = 0$ \\
	
	$\eta(g)$ & $\dim\cb{v \in V \;|\; gv = v}$ \\
	
	$\gcd$  & greatest common divisor  \\
	
	$\lcm$ & least common multiple \\
	
	$o_1$  & $o_1(x, y) := \pt{q - 1}/{\gcd\pt{\log_\varepsilon(x/y), q - 1}}$ where $x, y \in \BF_q^\times$  \\
	
	$o_2$ & $o_2(z) := {\log_\omega(z)}/{\gcd\pt{\log_\omega(z), q + 1}}$ where $z \in \BK_q^\times$ \\
	
	$c_\square$ & conjugacy classes, see Lemma \ref{conjugacy_class} \\
	
	$\fC_\square$ & conjugacy classes of derangements, see Definition \ref{derangement_conjugacy_class} \\
	
	$\sigma_\square$ & character-sums, see Definition \ref{derangement_conjugacy_class} \\
	
	$\jk{x, f}$ & the value of the function $f$ at $x$ \\
	
	$\binom{X}{k}$ & the family of $k$-element subsets of $X$ \\
	
	$g_1 {\sim} g_2$ & $g_1$ and $g_2$ are conjugate \\
	\bottomrule[0.1em]
\end{tabular}\]

\newpage
\section{Background}\label{Background}

\subsection{Representation theory}\label{representation_theory}

Let $G$ be a finite group. A \emph{$\BC$-representation} of $G$ is a pair $(\rho, V)$, where $V$ is a finite-dimensional $\BC$-vector space and $\rho: G \rightarrow \GL(V)$ is a homomorphism. An \emph{equivalence} between two representations $(\rho, V)$ and $(\rho', V')$ is a linear isomorphism $\psi: V \rightarrow V'$ such that $\psi(\rho(g)(v)) = \rho'(g)(\psi(v))$ for all $g \in G$ and $v \in V$. The representation $(\rho, V)$ is said to be \emph{irreducible} if there is no proper subspace of $V$ which is $\rho(g)$-invariant for all $g \in G$. There are only finitely many equivalence classes of irreducible $\BC$-representations of $G$. The \emph{character} $\chi_\rho$ of $\rho$ is the map defined by $\chi_\rho(g) = \mathrm{Tr}(\rho(g))$ where $\mathrm{Tr}$ denotes the trace. For more background, the readers may consult \cite{MR1878556, MR1824028, MR450380}.

$\GL_2(q)$ has four types of irreducible $\BC$-representations, namely the \textbf{O}ne-dimensional representations, the \textbf{S}pecial representations, the \textbf{P}rincipal series representations $V_{m, n}$, and the \textbf{W}eil representations. In particular, for $V_{m, n}$, the linear action of $\GL_2(q)$ on $V_{m, n}$ is the right regular representation, namely defined by $\jk{g_1, g_2(f)} := \jk{g_1g_2, f}$ where $g_1, g_2 \in \GL_2(q)$ and $f \in V_{m, n}$. For more background, the readers may consult \cite{MR1431508}.

\subsection{Cayley Graphs}

Let $X \subseteq G$ is inverse-closed, the Cayley graph $\Cay(G, X)$ is the graph with vertices $G$ and edges $\cb{\cb{u, v} \in \binom{G}{2}: uv^{-1} \in X}$.

\begin{theorem}[Babai \cite{MR546860}; Diaconis and Shahshahani \cite{MR626813}; Roichiman \cite{MR1400314}]\label{babai}
	
	Let $G$ be a finite group, let $X \subseteq G$ be inverse-closed and conjugation-invariant, and let $\Cay(G, X)$ be the Cayley graph. The eigenvalues of the Cayley graph $\Cay(G, X)$ are given by (where $\rho$ ranges over all irreducible $\BC$-representations of $G$)
	\[\theta_\rho = \frac{1}{\dim\rho}\cdot\sum_{x \in X}\chi_\rho(x) \text{ with multiplicity } \sum\limits_{\rho': \theta_{\rho'} = \theta_\rho}(\dim\rho')^2.\]

\end{theorem}

\subsection{Hoffman’s bound}

A weighted adjacency matrix $A_\Gamma$ for a graph $\Gamma$ is a real symmetric matrix, with zeros on the main diagonal, in which rows and columns are indexed by the vertices and the $(i, j)$-entry is $0$ if $i\not\sim j$.

\begin{theorem}[Hoffman \cite{MR284373}]\label{Hoffman}
	
	Let $\Gamma$ be a graph with vertex-set $V$, and let $A_\Gamma$ be a weighted adjacency matrix for $\Gamma$ with constant row sum. Let $\theta_1, \theta_2, \ldots, \theta_n$ be eigenvalues of $A_\Gamma$ ordered by descending absolute value. Let $S_1, S_2 \subseteq V$ such that there are no edges of $\Gamma$ between $S_1$ and $S_2$. Then
	\[\sqrt{\frac{\abs{S_1}}{\abs{V}}\cdot\frac{\abs{S_2}}{\abs{V}}} \ls \frac{\abs{\theta_2}}{\theta_1 + \abs{\theta_2}}.\]
	
\end{theorem}

\newpage
\section{Proof of Theorem \ref{main_thm}}\label{Proof}

\subsection{Strategy}

Note that $\dim\cb{v \in V: g_1v = g_2v} = \dim\cb{v \in V: g_2^{-1}g_1v = v}$, thus we define $\eta(g) := \dim\cb{v \in V: gv = v}$. With this terminology, two subsets $S_1, S_2 \subseteq G$ are cross-$t$-intersecting if and only if $\eta(g_2^{-1}g_1) \gs t$ for any $(g_1, g_2) \in S_1 \times S_2$. If we construct the Cayley graph $\Cay(G, X)$ with $X = \cb{g \in G: \eta(g) < t}$, then for two subsets $S_1, S_2 \subseteq G$, they are cross-$t$-intersecting if and only if no edges of $\Cay(G, X)$ between $S_1$ and $S_2$. (For now, let $A_G$ denote the weighted adjacency matrix of $\Cay(G, X)$.) This observation allows us, with the help of Theorem \ref{Hoffman}, to use the eigenvalues of $A_G$ to give an upper bound for $\abs{S_1}\cdot\abs{S_2}$, and Theorem \ref{babai} will be useful in calculating the eigenvalues of $A_G$. It is worth noting that the condition for applying Theorem \ref{Hoffman} is somewhat strict, namely $\theta_1$ and $\abs{\theta_2}$ should be distinct. In fact, we can ensure this by letting the multiplicity of $\theta_1$ as an eigenvalue of $A_G$ is exactly $1$.

Hence, the proof of Theorem \ref{main_thm} is divided into three steps. Step 1: to find the derangements, which is the generating set $X$ for the Cayley graph $\Cay(G, X)$. Step 2: to use Theorem \ref{babai} to calculate the eigenvalues of $A_G$. Step 3: to run a linear programming such that the multiplicity of $\theta_1$ is exactly $1$.

%From now on, recall that
%\[V_{m, n} = \cb{f \in \BC\fk{\GL_2(q)}: f\pt{\jz{a}{b}{}{d}x} = \mu_m(a)\cdot\mu_n(d)\cdot f(x)}\] 
%and $\eta(g) = \dim_\BC\cb{\varphi \in V_{m, n}: g.\varphi = \varphi}$.

\subsection{Find the derangements}

Note that $\eta$ is a class function, so we only need to compute the value of $\eta$ for one representative from each conjugacy class of $\GL_2(q)$. Recall that $\varepsilon$ is a fixed generator of $\BF_q^\times$, $\BK_q = \BF_q\fk{\sqrt{\varepsilon}}$ and $\omega$ is a fixed generator of $\BK_q^\times$ with $\CN\omega = \varepsilon$. Also recall that $\Theta: \BK_q^\times \hookrightarrow \GL_2(q)$ defined by $\Theta(x + y\sqrt{\varepsilon}) = \jz{x}{y}{y\varepsilon}{x}$. 

\begin{lemma}[Table 12.4 in \cite{MR1878556}]\label{conjugacy_class}
	The four conjugacy classes of $\GL_2(q)$ are shown in the table below.
	
	\[\begin{tabular}{p{5.5cm}p{3.15cm}p{4.35cm}}
		\toprule[0.1em]
		
		Class & Number of classes & Number of elements in it \\
		
		\midrule[0.1em]
		
		$c_1(x) := \jz{x}{}{}{x}$ & $q - 1$ & $1$ \\
		
		$c_2(x) := \jz{x}{}{1}{x}$ & $q - 1$ & $q^2 - 1$ \\
		
		$c_3(x, y) := \jz{x}{}{}{y}$ with $x \ne y$ & $\binom{q - 1}{2}$ & $q^2 + q$ \\
		
		$c_4(z) := \Theta(z)$ with $z \in {\BK_q^\times} \setminus \BF_q^\times$ & $\binom{q}{2}$ & $q^2 - q$ \\
		
		\bottomrule[0.1em]
	\end{tabular}\]
	
\end{lemma}

\begin{remark}
	
	$c_3(x, y) \sim c_3(y, x)$ and $c_4(z) \sim c_4(z^q)$.	
	
\end{remark}

Let $B := \cb{\jz{a}{b}{}{d} \in \GL_2(q)}$ and we have the index $\fk{\GL_2(q): B} = q + 1$. A $(q + 1)$-subset $T = \cb{x_1, \ldots, x_{q + 1}} \subseteq \GL_2(q)$ is said to be a \emph{$B$-transversal} if $\GL_2(q) = \biguplus_{i = 1}^{q + 1}Bx_i$. Let
\[T_1 = \cb{\jz{1}{}{x}{1}: x \in \BF_q} \cup \cb{\jz{}{1}{1}{}} \text{ and } T_2 = \cb{\Theta(\omega^i): 0 \ls i \ls q}.\]

\begin{lemma}
	
	$T_1$ and $T_2$ are both $B$-transversals of $\GL_2(q)$.	
	
\end{lemma}

\begin{proof}
	
	Since $\abs{T_i} = q + 1$, it suffices to show that $g_1g_2^{-1} \notin B$ for any $g_1, g_2 \in T_i$.
	
	For $T_1$, note that
	$\jz{1}{}{x_1}{1}\jz{1}{}{x_2}{1}^{-1} = \jz{1}{}{x_1 - x_2}{1} \notin B \text{ if } x_1 \ne x_2$
	and
	$\jz{1}{}{x}{1}\jz{}{1}{1}{}^{-1} = \jz{}{1}{1}{x} \notin B$.
	Hence $T_1$ is a $B$-transversal.
	
	For $T_2$, note that $B \cap \Theta\pt{\BK_q^\times} = \Theta(\BF_q^\times)$. If $0 \ls d_1 < d_2 \ls q$, then $1 \ls d_2 - d_1 \ls q$. Note that
	$\Theta(\omega^{d_1})^{-1}\Theta(\omega^{d_2}) = \Theta(\omega^{d_2 - d_1}) \notin B$ since $q + 1 \nmid d_2 - d_1$.
	Hence $T_2$ is also a $B$-transversal.\qedhere
	
\end{proof}

Recall that $\mu_m: \BF_q^\times \rightarrow \BC^\times$, $\mu_m(\varepsilon^i) = \exp\pt{2\cdot\pi\cdot\sqrt{-1}\cdot m \cdot i/(q - 1)}$,
\[V_{m, 0} := \cb{f: \GL_2(q) \rightarrow \BC \;|\; f\pt{\jz{a}{b}{}{d}x} = \mu_m(a) \cdot f(x)},\]
$\eta(g) := \dim_\BC\cb{f \in V_{m, 0} \;|\; g(f) = f}$ and $\one(x) = 1$ if $x = 1$, $\one(x) = 0$ if $x \ne 1$.

\begin{lemma}\label{eta_lemma}
	
	Suppose $\mathrm{char}(\BF_q) = p$, factor $q - 1 = \ell \cdot m$ where $\ell$ is an odd prime with $\ell \nmid m$. Set $o_1 = o_1(x, y) = \frac{q - 1}{\gcd\pt{\log_\varepsilon(x/y), q - 1}}$ where $x, y \in \BF_q^\times$ and $o_2 = o_2(z) = \frac{\log_\omega(z)}{\gcd\pt{\log_\omega(z), q + 1}}$ where $z \in \BK_q^\times$. Then we have
	\[\eta(g) = \left\{\begin{aligned}
		&(q + 1)\cdot\one\fk{{\mu_{m}(x)}},&&\text{ if } g = c_1(x);\\
		&\pt{{q}{p^{-1}} + 1}\cdot\one\fk{{\mu_{m}(x)}}, &&\text{ if } g = c_2(x);\\
		&\one\fk{\mu_m(x)} + \one\fk{\mu_m(y)} + \gcd{(\log_\varepsilon\pt{{x}/{y}}, q - 1)}\cdot\one\fk{\mu_{m}(x^{o_1})},&&\text{ if } g = c_3(x, y);\\
		&\gcd(\log_\omega(z), q + 1)\cdot\one\fk{\mu_{m}\pt{\varepsilon^{o_2}}},&&\text{ if } g = c_4(z).
	\end{aligned}\right.\]	
	
\end{lemma}

\begin{proof}
	
	The method for calculating $\eta$ is as follows: Note that $\dim_{\BC}V_{m, 0} = q + 1$ and for any $f \in V_{m, 0}$, $f$ is completely determined by $f(x_1), \ldots, f(x_{q + 1})$ where $\cb{x_1, \ldots, x_{q + 1}}$ is a $B$-transversal of $\GL_2(q)$. Fix $g \in \GL_2(q)$ and suppose that $g(\varphi) = \varphi$ where $\varphi \in V_{m, 0}$. Next we examine how many independent equations are needed to relate $f(x_1), \ldots, f(x_{q + 1})$, say $k$ equations, so $\eta(g) = q + 1 - k$. Because $\eta$ is a class function and $\GL_2(q)$ has four conjugacy classes, there are four cases to discuss. 
	
	\textbf{Case 1.} Suppose $\jz{x}{}{}{x}(\varphi) = \varphi$, then for any $t \in T_1$, we have
	\[\jk{t, \varphi}
	= \jk{t, \jz{x}{}{}{x}(\varphi)}
	= \jk{\jz{x}{}{}{x}t, \varphi}
	= \mu_{m}(x)\cdot\jk{t, \varphi}.\]
	Hence $\eta\pt{c_1(x)} = (q + 1)\cdot\one\fk{{\mu_{m}(x)}}$.
	
	\textbf{Case 2.} Suppose $\jz{x}{}{1}{x}(\varphi) = \varphi$, then for any $t \in \BF_q$, we have
	\[\begin{aligned}
		\jk{\jz{1}{}{t}{1}, \varphi}
		&= \jk{\jz{1}{}{t}{1}, \jz{x}{}{1}{x}(\varphi)}
		&&= \jk{\jz{1}{}{t}{1}\jz{x}{}{1}{x}, \varphi} \\
		&= \jk{\jz{x}{}{t\cdot x + 1}{x}, \varphi}
		&&= \jk{\jz{x}{}{}{x}\jz{1}{}{t + x^{-1}}{1}, \varphi} \\
		&= \mu_{m}(x)\cdot\jk{\jz{1}{}{t + x^{-1}}{1}, \varphi}
		&&= \mu_{m}(x^p)\cdot\jk{\jz{1}{}{t}{1}, \varphi}.
	\end{aligned}\]
	and
	\[\begin{aligned}
		\jk{\jz{}{1}{1}{}, \varphi} 
		&= \jk{\jz{}{1}{1}{}, \jz{x}{}{1}{x}(\varphi)}
		&= \jk{\jz{}{1}{1}{}\jz{x}{}{1}{x}, \varphi} \\
		&= \jk{\jz{1}{x}{x}{}, \varphi}
		&= \jk{\jz{x}{1}{}{x}\jz{}{1}{1}{}, \varphi} \\
		&= \mu_{m}(x) \cdot \jk{\jz{}{1}{1}{}, \varphi}.
	\end{aligned}\]
	Hence $\eta\pt{c_2(x)} = \frac{q}{p}\cdot\one\fk{{\mu_{m}(x^p)}} + \one\fk{{\mu_{m}(x)}} = \pt{q\cdot p^{-1} + 1}\cdot\one\fk{\mu_m(x)}$ since $p \nmid q - 1$.
	
	\textbf{Case 3.} Suppose $\jz{x}{}{}{y}(\varphi) = \varphi$ and the order of $x\cdot y^{-1}$ is $o_1$ in $\BF_q^\times$, explicitly, $o_1 = \frac{q - 1}{\gcd\pt{\log_\varepsilon(x/y), q - 1}}$. Then for any $t \in \BF_q$, we have
	\[\begin{aligned}
		\jk{\jz{1}{}{t}{1}, \varphi}
		&= \jk{\jz{1}{}{t}{1}, \jz{x}{}{}{y}(\varphi)} 
		&&= \jk{\jz{1}{}{t}{1}\jz{x}{}{}{y}, \varphi} \\
		&= \jk{\jz{x}{}{t\cdot x}{y}, \varphi} 
		&&= \jk{\jz{x}{}{}{y}\jz{1}{}{t\cdot x\cdot y^{-1}}{1}, \varphi} \\
		&= \mu_m(x)\cdot\jk{\jz{1}{}{t\cdot x\cdot y^{-1}}{1}, \varphi}  
		&&= \mu_m(x^{o_1})\cdot\jk{\jz{1}{}{t}{1}, \varphi}.
	\end{aligned}\]
	and
	\[\begin{aligned}
		\jk{\jz{}{1}{1}{}, \varphi}
		&= \jk{\jz{}{1}{1}{}, \jz{x}{}{}{y}(\varphi)}
		&= \jk{\jz{}{1}{1}{}\jz{x}{}{}{y}, \varphi} \\
		&= \jk{\jz{}{y}{x}{}, \varphi}
		&= \jk{\jz{y}{}{}{x}\jz{}{1}{1}{}, \varphi} \\
		&= \mu_m(y)\cdot\jk{\jz{}{1}{1}{}, \varphi}.
	\end{aligned}\]
	Hence $\eta\pt{c_3(x, y)} = \one\fk{\mu_m(x)} + \one\fk{\mu_m(y)} + \frac{q - 1}{o_1}\cdot\one\fk{\mu_{m}(x^{o_1})} = \one\fk{\mu_m(x)} + \one\fk{\mu_m(y)} + \gcd\pt{\log_\varepsilon(x/y), q - 1}\cdot\one\fk{\mu_m(x^{o_1})}$.
	
	\textbf{Case 4.} Suppose $\Theta(z)(\varphi) = \varphi$, then for any $0 \ls i \ls q$, we have (we omit the notation $\Theta$)
	\[\begin{aligned}
		\jk{\omega^i, \varphi}
		&= \jk{\omega^i, \omega^{\log_\omega(z)}(\varphi)}
		&&= \jk{\omega^i \cdot \omega^{\log_\omega(z)}, \varphi} \\
		&= \jk{\omega^{i + \log_\omega(z)}, \varphi} 
		&&= \jk{\omega^{i + \lcm(\log_\omega(z), q + 1)}, \varphi} \\
		&= \jk{\pt{\omega^{q + 1}}^{o_2}\cdot\omega^i, \varphi}
		&&= \jk{\varepsilon^{o_2}\cdot\omega^i, \varphi} \\
		&= \mu_{m}(\varepsilon^{o_2})\cdot\jk{\omega^i, \varphi}.
	\end{aligned}\]
	Hence $\eta(c_4(z)) = \gcd\pt{\log_\omega(z), q + 1}\cdot\one\fk{{\mu_{m}(\varepsilon^{o_2})}}$.\qedhere
	
\end{proof}

%From now on we let $n = 0$.

\begin{corollary}\label{derangements}
	
	Factor $q - 1 = \ell \cdot m$ where $\ell$ is an odd prime with $\ell \nmid m$. Then
	\[\eta(g) = 0 \text{ iff } \left\{\begin{aligned}
		&\log_\varepsilon(x) \not\equiv 0 &\pmod \ell,&\quad&&\text{ if } g = c_1(x) \text{ or } c_2(x);\\
		&\log_\varepsilon(x) \equiv \log_\varepsilon(y) \not\equiv 0 &\pmod \ell,&\quad&&\text{ if } g = c_3(x, y);\\
		&\log_\omega(z) \not\equiv 0 &\pmod \ell,&\quad&&\text{ if } g = c_4(z).
	\end{aligned}\right.\]	
	
\end{corollary}

\begin{proof}
	
	Recall that $o_1 = o_1(x, y) = \frac{q - 1}{\gcd\pt{\log_\varepsilon(x/y), q - 1}}$ and $o_2 = o_2(z) = \frac{\log_\omega(z)}{\gcd\pt{\log_\omega(z), q + 1}}$.
	
	\textbf{Case 1 and 2.} If $g = c_1(x)$ or $c_2(x)$, then
	\[\eta(g) \ne 0 \text{ iff } \mu_m(x) = 1 \text{ iff } q - 1 \mid m \cdot \log_\varepsilon(x) \text{ iff } \ell \mid \log_\varepsilon(x).\]
	
	\textbf{Case 3.} If $g = c_3(x, y)$, then
	\[\begin{aligned}
		\eta(g) \ne 0
		&\text{ iff } \mu_{m}(x^{o_1}) = 1 
		\text{ iff } q - 1 \mid o_1 \cdot m\cdot \log_\varepsilon(x)\\
		&\text{ iff } q - 1 \mid \frac{(q - 1) \cdot m \cdot \log_\varepsilon(x)}{\gcd(\log_\varepsilon(x/y), q - 1)}
		\text{ iff } \gcd(\log_\varepsilon(x/y), q - 1) \mid m\cdot \log_\varepsilon(x)\\
		&\text{ iff } \left\{\begin{aligned}
			&\gcd\pt{\log_\varepsilon(x/y), m}\mid m\cdot \log_\varepsilon(x),&\quad&\text{ if } \ell \nmid \log_\varepsilon(x/y);\\
			&\ell\cdot\gcd\pt{\frac{\log_\varepsilon(x/y)}{\ell}, m}\mid m \cdot \log_\varepsilon(x),&\quad&\text{ if } \ell \mid \log_\varepsilon(x/y).
		\end{aligned}\right.\\
		&\text{ iff } \log_\varepsilon(x) \not\equiv \log_\varepsilon(y) \text{ or } \log_\varepsilon(x) \equiv \log_\varepsilon(y) \equiv 0 \pmod{\ell}, \quad \text{ since } \ell \nmid m.
	\end{aligned}\]
	
	\textbf{Case 4.} If $g = c_4(z)$,
	\[\begin{aligned}
		\eta(g) \ne 0
		&\text{ iff } \mu_m\pt{\varepsilon^{o_2}} = 1 
		\text{ iff } q - 1 \mid \frac{m \cdot \log_\omega(z)}{\gcd(\log_\omega(z), q + 1)} \\
		&\text{ iff } \ell \mid \frac{\log_\omega(z)}{\gcd(\log_\omega(z), q + 1)} 
		\text{ iff } \ell \mid \log_\omega(z), \quad \text{ since } \ell \nmid q + 1.
	\end{aligned}\]
	The proof is completed.\qedhere	
\end{proof}

Let that $H := \cb{g \in \GL_2(q) \;|\; \log_\varepsilon(\det(g)) \equiv 0 \pmod{\ell}}$.

\begin{corollary}
	
	$H$ is $1$-intersecting, that is, $\eta(g) \gs 1$ for any $g \in H$.
	
\end{corollary}

\begin{proof}
	
	\textbf{Case 1 and 2.} If $g = c_1(x) \in H$ or $g = c_2(x) \in H$, then $\log_\varepsilon(\det(g)) = \log_\varepsilon(x^2) = 2\cdot\log_\varepsilon(x) \equiv 0 \pmod{\ell}$, which is equivalent to $\log_\varepsilon(x) \equiv 0 \pmod{\ell}$ since $\ell \nmid 2$. Hence $\eta(g) \gs 1$ by Corollary \ref{derangements}.
	
	\textbf{Case 3.} If $g = c_3(x, y) \in H$, then $\log_\varepsilon(\det(g)) = \log_\varepsilon(x) + \log_\varepsilon(y) \equiv 0 \pmod{\ell}$. If $\log_\varepsilon(x) \equiv \log_\varepsilon(y) \pmod{\ell}$, then $\log_\varepsilon(\det(g)) = 2 \cdot \log_\varepsilon(x) \equiv 0 \pmod{\ell}$. Hence $\eta(g) \gs 1$ by Corollary \ref{derangements}.
	
	\textbf{Case 4.} If $g = c_4(z) \in H$, then $\log_\varepsilon(\det(g)) = \log_\varepsilon(\CN z) = \log_{\omega^{q + 1}}(z^{q + 1}) = \log_\omega(z) \equiv 0 \pmod{\ell}$. Hence $\eta(g) \gs 1$ by Corollary \ref{derangements}.\qedhere
	
\end{proof}

Hence $\max_{\pt{S_1, S_2}}\sqrt{\abs{S_1}\cdot\abs{S_2}} \gs \abs{H} = \abs{\GL_2(q)}/\ell$, where $S_1, S_2 \subseteq \GL_2(q)$ are cross-$1$-intersecting.

%\newpage
\subsection{Compute the eigenvalues}

A \emph{character table} is a table whose rows are irreducible characters, and whose columns are conjugacy classes of a group. Let $\zeta = \exp\pt{\frac{2\cdot\pi\cdot\sqrt{-1}}{q - 1}}$ and $\xi = \exp\pt{\frac{2\cdot\pi\cdot\sqrt{-1}}{q^2 - 1}}$. Let $\mu_i: \BF_q^\times \rightarrow \BC^\times$, $\mu_i(\varepsilon^j) = \zeta^{i\cdot j}$ and $\lambda_i: \BK_q^\times \rightarrow \BC^\times$, $\lambda_i(\omega^j) = \xi^{i\cdot j}$.

\begin{theorem}[Table 12.5(I)-(IV) in \cite{MR1878556}]
	
	The character table of $\GL_2(q)$ is:
	
	\[\begin{tabular}{p{0.4cm}|p{2.7cm}p{2.7cm}p{4.5cm}p{2.7cm}}
		\toprule[0.1em]
		
		& $c_1(x)$ & $c_2(x)$ & $c_3(x, y)$ & $c_4(z)$ \\
		
		\midrule[0.1em]
		
		\textbf{O} & $\mu_{2r}(x)$ & $\mu_{2r}(x)$ & $\mu_{r}(x\cdot y)$ & $\mu_r(\CN z)$ \\
		
		\textbf{S} & $q\cdot\mu_{2r}(x)$ & $0$ & $\mu_{r}(x\cdot y)$ & $-\mu_r(\CN z)$ \\
		
		\textbf{P} & $(q + 1)\cdot\mu_{r + s}(x)$ & $\mu_{r + s}(x)$ & $\mu_r(x)\cdot\mu_s(y) + \mu_r(y)\cdot\mu_s(x)$ & $0$ \\
		
		\textbf{W} & $(q - 1)\cdot\lambda_{r}(x)$ & $-\lambda_{r}(x)$ & $0$ & $-\lambda_r(z) - \lambda_r(z^q)$ \\
		
		\bottomrule[0.1em]
	\end{tabular}\]
	In \textbf{O}-row and \textbf{S}-row, the range of $r$ is $1 \ls r \ls q - 1$; in \textbf{P}-row, the range of $r, s$ is $1 \ls r < s \ls q - 1$; in \textbf{W}-row, the range of $r$ is $1 \ls r \ls q^2 - 1$ with $q + 1 \nmid r$.
	
\end{theorem}

\begin{definition}\label{derangement_conjugacy_class}
	
	Factor $q - 1 = \ell \cdot m$ where $\ell$ is an odd prime with $\ell \nmid m$, define
	
	\begin{itemize}
		\item $\sigma_C(r) := \sum_{x \in C}\mu_r(x)$, where $C := \cb{x \in \BF_q^\times \;|\; \log_\varepsilon(x) \not\equiv 0 \pmod{\ell}}.$ 
		
		\item $\sigma_D(r, s) := \sum_{\pt{x, y} \in D}\mu_r(x)\cdot\mu_s(y)$, where $$D := \cb{\pt{x, y} \in \pt{\BF_q^\times}^{2} \;|\; x \ne y \text{ and }  \log_\varepsilon(x) \equiv \log_\varepsilon(y) \not\equiv 0\pmod{\ell}}.$$Note that, by symmetry, we have$$\sum_{\cb{x, y} \in \binom{\BF_q^\times}{2}: \log_\varepsilon(x) \equiv \log_\varepsilon(y) \not\equiv 0 \pmod{\ell}}\mu_r(x)\cdot\mu_s(y) + \mu_r(y)\cdot\mu_s(x) = \sigma_D(r, s).$$
		
		\item $\sigma_E^{(1)}(r) := \sum_{z \in E}\lambda_r(z)$, $\sigma_E^{(2)}(r) := \sum_{z \in E}\mu_r\pt{\CN z}$, where $$E := \cb{z \in \BK_q^\times \setminus \BF_q^\times \;|\;  \log_\omega(z) \not\equiv 0\pmod{\ell}}.$$
	\end{itemize}
	
\end{definition}

From Lemma \ref{derangements} we know that the derangements of $\GL_2(q)$ belong to four families of conjugacy classes:
\begin{itemize}
	\item $\fC_i = \cb{g \in \GL_2(q) \;|\; g {\sim} c_i(x) \text{ with } \log_\varepsilon(x) \not\equiv 0 \pmod{\ell}}$, $i = 1, 2$;
	
	\item $\fC_3 = \cb{g \in \GL_2(q) \;|\; g {\sim} c_3(x, y) \text{ with }  \log_\varepsilon(x) \equiv \log_\varepsilon(y) \not\equiv 0 \pmod{\ell}}$;
	
	\item $\fC_4 = \cb{g \in \GL_2(q) \;|\; g {\sim} c_4(z) \text{ with }  \log_\omega(z) \not\equiv 0 \pmod{\ell}}$.
\end{itemize}

Note that $\lambda_r|_{\BF_q} = \mu_r$. Combining Theorem \ref{babai}, Lemma \ref{conjugacy_class} and Definition \ref{derangement_conjugacy_class} we have the following.

\begin{lemma}
	The eigenvalues for the four Cayley graphs are shown as follows:
	\[\begin{tabular}{p{0.4cm}|p{3.1cm}p{3.1cm}p{3.3cm}p{3.1cm}}
		\toprule[0.1em]
		
		& $\Cay(\GL_2(q), \fC_1)$ & $\Cay(\GL_2(q), \fC_2)$ & $\Cay(\GL_2(q), \fC_3)$ & $\Cay(\GL_2(q), \fC_4)$ \\
		
		\midrule[0.1em]
		
		\textbf{O} & $\sigma_C(2r)$ & $(q^2 - 1)\cdot\sigma_C(2r)$ & $q\cdot(q + 1)\cdot\sigma_D(r, r)/2$ & $q\cdot(q - 1)\cdot\sigma_E^{(2)}(r)/2$ \\
		
		\textbf{S} & $\sigma_C(2r)$ & $0$ & $(q + 1)\cdot\sigma_D(r, r)/2$ & $-(q - 1)\cdot\sigma_E^{(2)}(r)/2$ \\
		
		\textbf{P} & $\sigma_C(r + s)$ & $(q - 1)\cdot\sigma_C(r + s)$ & $q\cdot\sigma_D(r, s)$ & $0$ \\
		
		\textbf{W} & $\sigma_C(r)$ & $-(q + 1)\cdot\sigma_C(r)$ & $0$ & $-q\cdot\sigma_E^{(1)}(r)$ \\
		
		\bottomrule[0.1em]
	\end{tabular}\]
\end{lemma}

\begin{lemma}\label{sum1}
	
	Put $u := r + s$. Ignoring the zero-cases, we have the following
	\[\begin{tabular}{p{3.7cm}p{2.8cm}|p{3.7cm}p{2.8cm}}
		\toprule[0.1em]
		
		Cases & $m^{-1}\cdot\sigma_C(r)$ & Cases & $m^{-1}\cdot\sigma_E^{(1)}(r)$ \\
		
		\midrule[0.1em]
		
		$q - 1 \mid r$ & $\ell - 1$ & $q - 1 \mid r$ & $-(\ell - 1)$ \\
		
		$m \mid r$; $\ell \nmid r$ & $-1$ & $m \mid r$; $\ell \nmid r$ & $1$ \\
		
		\bottomrule[0.1em]
		\toprule[0.1em]
		
		Cases & $m^{-1}\cdot\sigma_D(r, s)$ & Cases & $m^{-1}\cdot\sigma_E^{(2)}(r)$ \\
		
		\midrule[0.1em]
		
		$m \mid r, s$; $\ell \mid u$ & $(\ell - 1)\cdot(m - 1)$ & $q - 1 \mid r$ & $q\cdot(\ell - 1)$ \\
		
		$m \mid r, s$; $\ell \nmid u$ & $-(m - 1)$ & $m \mid r$; $\ell \nmid r$ & $-q$ \\
		
		$m \nmid r, s$; $q - 1 \mid u$ & $-(\ell - 1)$ & $\frac{q - 1}{2} \mid r$; $m \nmid r$ & $-(\ell - 1)$ \\
		
		$m \nmid r, s$; $m \mid u$; $\ell \nmid u$ & $1$ & $\frac{m}{2} \mid r$; $m, \ell \nmid r$ & $1$ \\
		
		\bottomrule[0.1em]
	\end{tabular}\]	
\end{lemma}

\begin{proof}
	Note that, if $e_k = \exp\pt{2\cdot\pi\cdot\sqrt{-1}/k}$, then we have the easy fact that
	\[\sum_{i = 1}^{k}e_k^{i\cdot r} = \left\{\begin{aligned}
		&k,&\qquad&\text{ if } k \mid r;\\
		&0,&\qquad&\text{ if } k \nmid r.
	\end{aligned}\right.\] 
	We will repeatly use this fact in the calculation below.
	\[\begin{aligned}
		\sigma_C(r)
		&= \sum_{i = 1}^{q - 1}\mu_r\pt{\varepsilon^i} - \sum_{i = 1}^{m}\mu_r\pt{\varepsilon^{i\cdot\ell}} \\
		&= \sum_{i = 1}^{q - 1}{\zeta^{i\cdot r}} - \sum_{i = 1}^{m}{\zeta^{i\cdot \ell \cdot r}} \\
		&= \left\{\begin{aligned}
			&m\cdot(\ell - 1),&\quad&\text{if } q - 1 \mid r;\\
			&m\cdot(-1),&\quad&\text{if } m \mid r \text{ and } \ell \nmid r,
		\end{aligned}\right.\\
		%%%%%%%%%%%%%%%%%%%%
		\sigma_D(r, s)
		&= \sum_{i = 1}^{q - 1}\sum_{j = 1}^{m}\mu_r\pt{\varepsilon^i}\cdot\mu_s\pt{\varepsilon^{i + j\cdot\ell}} - \sum_{i = 1}^{m}\sum_{j = 1}^{m}\mu_r\pt{\varepsilon^{i\cdot\ell}}\cdot\mu_s\pt{\varepsilon^{j\cdot\ell}}\\
		&\quad - \sum_{i = 1}^{q - 1}\mu_r\pt{\varepsilon^i}\cdot\mu_s\pt{\varepsilon^{i}} + \sum_{i = 1}^{m}\mu_r\pt{\varepsilon^{i\cdot\ell}}\cdot\mu_s\pt{\varepsilon^{i\cdot\ell}}\\ 
		&= \sum_{i = 1}^{q - 1}\sum_{j = 1}^{m}\mu_{u}\pt{\varepsilon^i}\cdot\mu_s\pt{\varepsilon^{j\cdot\ell}} - \sum_{i = 1}^{m}\sum_{j = 1}^{m}\mu_r\pt{\varepsilon^{i\cdot\ell}}\cdot\mu_s\pt{\varepsilon^{j\cdot\ell}} - \sum_{i = 1}^{q - 1}\mu_{u}\pt{\varepsilon^i} + \sum_{i = 1}^{m}\mu_{u}\pt{\varepsilon^{i\cdot\ell}} \\
		&= {\sum_{i = 1}^{q - 1}{\zeta^{i \cdot u}} \cdot \sum_{j = 1}^{m}{\zeta^{j \cdot \ell \cdot s}} - \sum_{i = 1}^{m}{\zeta^{i \cdot \ell \cdot r}} \cdot \sum_{j = 1}^{m}{\zeta^{j \cdot \ell \cdot s}} - \sum_{i = 1}^{q - 1}{\zeta^{i\cdot u}} + \sum_{i = 1}^{m}{\zeta^{i \cdot \ell \cdot u}}} \\
		&= \left\{\begin{aligned}
			&m\cdot(m - 1)\cdot(\ell - 1),&\quad&\text{if } m \mid r, s \text{ and } \ell \mid u;\\
			&-m\cdot(m - 1),&\quad&\text{if } m \mid r, s \text{ and } \ell \nmid u;\\
			&-m\cdot(\ell - 1),&\quad&\text{if } m \nmid r, s \text{ and } q - 1 \mid u;\\
			&m,&\quad&\text{if } m \nmid r, s \text{ and } m \mid u \text{ and } \ell \nmid u,
		\end{aligned}\right.\\
		%%%%%%%%%%%%%%%%%%%%
		\sigma_E^{(1)}(r)
		&= {\sum_{i = 1}^{q^2 - 1}\lambda_r\pt{\omega^i} - \sum_{i = 1}^{q - 1}\lambda_r\pt{\omega^{i\cdot(q + 1)}} - \sum_{i = 1}^{m\cdot(q + 1)}\lambda_r\pt{\omega^{i\cdot\ell}} + \sum_{i = 1}^{m}\lambda_r\pt{\omega^{i\cdot(q + 1)\cdot\ell}}}\\
		&= {\sum_{i = 1}^{q^2 - 1}{\xi^{i\cdot r}} - \sum_{i = 1}^{q - 1}{\xi^{i\cdot(q + 1)\cdot r}} - \sum_{i = 1}^{m\cdot(q + 1)}{\xi^{i\cdot\ell\cdot r}} + \sum_{i = 1}^{m}{\xi^{i\cdot(q + 1)\cdot\ell\cdot r}}}\\
		&= {-\sum_{i = 1}^{q - 1}{\xi^{i\cdot(q + 1)\cdot r}} + \sum_{i = 1}^{m}{\xi^{i\cdot(q + 1)\cdot\ell\cdot r}}}\qquad\text{since } q + 1 \nmid r\\
		&= \left\{\begin{aligned}
			&m\cdot(-\ell + 1),&\quad&\text{if } q - 1 \mid r;\\
			&m,&\quad&\text{if } m \mid r \text{ and } \ell \nmid r,
		\end{aligned}\right.
	\end{aligned}\]
	%%%%%%%%%%%%%%%%%%%%
	\[\begin{aligned}
		\sigma_E^{(2)}(r)
		&= {\sum_{i = 1}^{q^2 - 1}\mu_r\pt{\CN\omega^i} - \sum_{i = 1}^{q - 1}\mu_r\pt{\CN\omega^{i\cdot(q + 1)}} - \sum_{i = 1}^{m\cdot(q + 1)}\mu_r\pt{\CN\omega^{i\cdot\ell}} + \sum_{i = 1}^{m}\mu_r\pt{\CN\omega^{i \cdot (q + 1) \cdot \ell}}} \\
		&= {\sum_{i = 1}^{q^2 - 1}\mu_r\pt{\varepsilon^i} - \sum_{i = 1}^{q - 1}\mu_r\pt{\varepsilon^{i \cdot (q + 1)}} - \sum_{i = 1}^{m\cdot(q + 1)}\mu_r\pt{\varepsilon^{i \cdot \ell}} + \sum_{i = 1}^{m}\mu_r\pt{\varepsilon^{i \cdot (q + 1) \cdot \ell}}}\\
		&= {\sum_{i = 1}^{q^2 - 1}{\zeta^{i \cdot r}} - \sum_{i = 1}^{q - 1}{\zeta^{i \cdot 2 \cdot r}} - \sum_{i = 1}^{m\cdot(q + 1)}{\zeta^{i \cdot \ell \cdot r}} + \sum_{i = 1}^{m}{\zeta^{i \cdot 2 \cdot \ell \cdot r}}} \\
		&= \left\{\begin{aligned}
			&m\cdot q \cdot (\ell - 1),&\quad&\text{if } q - 1 \mid r;\\
			&-m\cdot q,&\quad&\text{if } m \mid r \text{ and } \ell \nmid r;\\
			&-m\cdot(\ell - 1),&\quad&\text{if } \frac{q - 1}{2} \mid r \text{ and } m \nmid r;\\
			&m,&\quad&\text{if } \frac{m}{2} \mid r \text{ and } m, \ell \nmid r.
		\end{aligned}\right.
	\end{aligned}\]
	The proof is completed.\qedhere
	
\end{proof}

\begin{lemma}\label{eigenvalue}
	
	Put $u := r + s$. The eigenvalues for the four weighted Cayley graphs $\Gamma_i := m^{-1}\cdot\Cay(\GL_2(q), \fC_i)$ ($i = 1, 2, 3, 4$) are shown as follows (detailed version):
	\[\begin{tabular}{p{3.95cm}|p{0.75cm}p{2.6cm}p{3.4cm}p{2.3cm}}
		\toprule[0.1em]
		
		&$\Gamma_1$ &$\Gamma_2$ &$\Gamma_3$ &$\Gamma_4$	\\
		
		\midrule[0.1em]
		
		\textbf{O}: $q - 1 \mid r$	&$\ell - 1$	&$(q^2 - 1)\cdot(\ell - 1)$	&$\binom{q + 1}{2}\cdot(\ell - 1)\cdot(m - 1)$	&$\binom{q}{2}\cdot q\cdot(\ell - 1)$\\
		
		\textbf{O}: $m \mid r$; $\ell \nmid r$	&$-1$	&$-(q^2 - 1)$	&$-\binom{q + 1}{2}\cdot(m - 1)$	&$-\binom{q}{2}\cdot q$\\
		
		\textbf{O}: $\frac{q - 1}{2} \mid r$; $m \nmid r$	&$\ell - 1$	&$(q^2 - 1)\cdot(\ell - 1)$	&$-\binom{q + 1}{2}\cdot(\ell - 1)$	&$-\binom{q}{2}\cdot(\ell - 1)$\\
		
		\textbf{O}: $\frac{m}{2} \mid r$; $m, \ell \nmid r$	&$-1$	&$-(q^2 - 1)$	&$\binom{q + 1}{2}$	&$\binom{q}{2}$\\
		
		\hline
		
		\textbf{S}: $q - 1 \mid r$	&$\ell - 1$	&$0$	&$\frac{q + 1}{2}\cdot(\ell - 1)\cdot(m - 1)$	&$-\binom{q}{2}\cdot(\ell - 1)$\\
		
		\textbf{S}: $m \mid r$; $\ell \nmid r$	&$-1$	&$0$	&$-\frac{q + 1}{2}\cdot(m - 1)$	&$\binom{q}{2}$\\
		
		\textbf{S}: $\frac{q - 1}{2} \mid r$; $m \nmid r$	&$\ell - 1$	&$0$	&$-\frac{q + 1}{2}\cdot(\ell - 1)$	&$\frac{q - 1}{2}\cdot(\ell - 1)$\\
		
		\textbf{S}: $\frac{m}{2} \mid r$; $m, \ell \nmid r$	&$-1$	&$0$	&$\frac{q + 1}{2}$	&$-\frac{q - 1}{2}$\\
		
		\hline
		
		\textbf{P}: $m \mid r, s$; $\ell \mid u$	&$\ell - 1$	&$(q - 1)\cdot(\ell - 1)$	&$q\cdot(\ell - 1)\cdot(m - 1)$	&$0$\\
		
		\textbf{P}: $m \mid r, s$; $\ell \nmid u$	&$-1$	&$-(q - 1)$	&$-q\cdot(m - 1)$	&$0$\\
		
		\textbf{P}: $m \nmid r, s$; $q - 1 \mid u$	&$\ell - 1$	&$(q - 1)\cdot(\ell - 1)$	&$-q\cdot(\ell - 1)$	&$0$\\
		
		\textbf{P}: $m \nmid r, s$; $m \mid u$; $\ell \nmid u$	&$-1$	&$-(q - 1)$	&$q$	&$0$\\
		
		\hline
		
		\textbf{W}: $q - 1 \mid r$	&$\ell - 1$	&$-(q + 1)\cdot(\ell - 1)$	&$0$	&$q\cdot(\ell - 1)$\\
		
		\textbf{W}: $m \mid r$; $\ell \nmid r$	&$-1$	&$q + 1$	&$0$	&$-q$\\
		
		\bottomrule[0.1em]
	\end{tabular}\]
	
\end{lemma}

\begin{remark}
	
	Note that the odd-numbered rows in the table above are $-(\ell - 1)$ times the next row.	
	
\end{remark}

\subsection{Run the linear programming}
Extracting the even-numbered rows from the table in Lemma \ref{eigenvalue} and multiplying them by $-1$ yields the following matrix. 
\[\mathcal{L} := \begin{pmatrix}
	1	&q^2 - 1	&\binom{q + 1}{2}\cdot(m - 1)	&\binom{q}{2}\cdot q\\
	
	1	&q^2 - 1	&-\binom{q + 1}{2}	&-\binom{q}{2}\\
	
	1	&0	&\pt{q + 1}\cdot(m - 1)/2	&-\binom{q}{2}\\
	
	1	&0	&-\pt{q + 1}/{2}	&\pt{q - 1}/{2}\\
	
	1	&q - 1	&q\cdot(m - 1)	&0\\
	
	1	&q - 1	&-q	&0\\
	
	1	&-(q + 1)	&0	&q
\end{pmatrix} \in \BR^{7 \times 4}.\]
In other words, after an rearrangement of the rows in the table in Lemma \ref{eigenvalue}, the result table becomes $\begin{pmatrix}
	(\ell - 1)\cdot\mathcal{L}\\-\mathcal{L}
\end{pmatrix}$.

Now we run a linear programming to find a weighting $w = \pt{2\cdot m \cdot \abs{\GL_2(q)}}^{-1}\cdot\pt{w_1, w_2, w_3, w_4}^\mathtt{T} \in \BR^{4 \times 1}$ where
\begin{itemize}
	\item $w_1 = -(q-1)\cdot \left(m^2 \cdot \left(2 \cdot q^2+2 \cdot q+1\right)-2 \cdot m-q^2+1\right)$;
	
	\item $w_2 = (m+q-1)^2$;
	
	\item $w_3 = 2\cdot(q-1)\cdot (m+q-1)$;
	
	\item $w_4 = 2\cdot m\cdot(m+q-1)$.
\end{itemize}
We check that
\[\mathcal{L}\cdot w = \begin{pmatrix}
	1 & \frac{-1}{\ell} & \frac{-1}{\ell} & \frac{-1}{\ell} & \frac{-(q - \ell)}{(q + 1)\cdot\ell} & \frac{-1}{\ell} & \frac{-1}{\ell}
\end{pmatrix}^{\mathtt{T}} \in \BR^{7\times1}\]

It is easy to see that $\frac{(q - \ell)\cdot(\ell - 1)}{(q + 1)\cdot\ell} < 1$. Let $G = \GL_2(q)$ and $$A_G := m^{-1}\cdot\sum\limits_{i = 1}^{4}w_i\cdot\Cay(G, \fC_i).$$ Let $\theta_1, \theta_2, \ldots, \theta_{\abs{G}}$ be eigenvalues of $A_G$ ordered by descending absolute value. By Theorem \ref{babai} and the calculation above, we have $\theta_1 = \ell - 1$ with multiplicity one, thus $\theta_2 = -1$. Now the Hoffman's ratio bound (Theorem \ref{Hoffman}) on this weighted adjacency matrix $A_G$ gives
\[\sqrt{\abs{S_1}\cdot\abs{S_2}} \ls \frac{\abs{-1}}{(\ell - 1) + \abs{-1}}\cdot\abs{\GL_2(q)} = \abs{\GL_2(q)}/\ell,\]
where $S_1, S_2 \subseteq \GL_2(q)$ are cross-$1$-intersecting, and the proof is complete.

\newpage
\bibliographystyle{abbrv}
\nocite{*}
\bibliography{EPS_refs}
	
\end{document}